\documentclass[12pt]{article}
\usepackage{amsmath,amsthm,hyperref}

\newtheorem{lem}{Lemma}
\newtheorem{thm}{Theorem}

\newtheorem{cor}{Corollary}

\begin{document}

\title{The Independent Chip Model and Risk Aversion}
\author{George T.~Gilbert\\ Texas Christian University\\
g.gilbert@tcu.edu}
\date{November 2009}
\maketitle

\begin{abstract}
We consider the Independent Chip Model (ICM) for expected value in
poker tournaments.  Our first result is that participating in a fair
bet with one other player will always lower one's expected value under
this model.  Our second result is that the expected value for players
not participating in a fair bet between two players always increases.
We show that neither result necessarily holds for a fair bet among
three or more players.
\end{abstract}

\thispagestyle{empty}

\section{Introduction}

The analysis of expected value in poker tournaments is more complex than for cash 
games. By this, we mean that chips in a cash game are equivalent to 
cash. On the other hand, the expected value of chips in a tournament 
are related to expected value in cash in a nonlinear way.

In a typical (freezeout) tournament, each player begins with the same
number of chips.  Once a player is out of chips, he or she is out of
the tournament.  Play continues until one player has all of the chips.
In most tournaments, however, the winning player gets only a portion
of the prize money.  The last player eliminated gets second place
money, the next-to-last eliminated gets third place money, and so
forth.  Effectively, the winner is forced to give back some of the
chips he or she has won.  Consequently, most players should play in a
somewhat risk averse manner, the extent depending on his or her
ability.

\section{Modeling Poker Tournaments}

In models of poker tournaments where all players have ``equal
abilities'' and ``equal opportunities,'' the probability a player
finishes first equals the fraction of the total chips in play that he
or she holds.  The model where a player wins or loses a single chip
with probability $1/2$ each is the standard Gambler's Ruin or
random walk problem going back to Huygens
\hyperlink{huy}{\cite{huy}}.  In fact, he further considers
constant, but unequal, probabilities of winning or losing, which can be
interpreted as introducing skill into the model.  See also
\hyperlink{feller}{\cite{feller}}.

The probability of finishing first equals the fraction of the total
chips in play held by the player much more generally, for instance if
a player's expected gain in chips is zero for each hand. If the 
player's proportion of all chips is $x$ and probability of winning 
the tournament is $f(x)$, then
$$(1-f(x))(-x)+f(x)(1-x)=0,$$
from which we see $f(x)=x$. Henke \hyperlink{hen}{\cite{hen}} 
tested this model on data from World Poker Tour final tables. There was 
reasonable agreement between theory and data. Nevertheless, 
the model modestly overestimated the probabilities of small stacks winning and 
modestly underestimated the probabilities of large stacks winning. This could 
be due to flaws in the model or to differences in the skill of 
players that led to the disparities in stack sizes.

In contrast, the probability of finishing second, third, and so forth
is very dependent on the model.  Even in the particular model where
two players are chosen at random and each wins or loses a single chip
from the other with probability $1/2$, a player's probability of
finishing second will depend not just on the fraction of chips held by
each player but on the actual number of chips held by the players.  In
a more general setting, Swan and Bruss \hyperlink{sw-br}{\cite{sw-br}}
give a recursive solution for the probability a particular player is
the first to go broke in terms of Markov processes and unfolding.
Unfortunately, it is not a computationally practical method for the
repeated calculations needed to analyze poker tournaments.  In the
case of three players, the problem is a discrete Dirichlet problem on
a triangle.  Ferguson \hyperlink{tferg}{\cite{tferg}} solved the
players' probabilities of finishing first, second, and third for
Brownian motion --- the limit as the number of chips increases to
infinity.  Employing a Schwarz-Christoffel transformation, he
expresses the answer in terms of the inverse of the incomplete beta
function, so it is not easy to actually compute these limiting
probabilities.  Due to the specialized techniques, even this does not
generalize to more than three players \hyperlink{blt}{\cite{blt}}.

In the early to middle stages of a tournament, the primary factor in
determining a player's expected cash winnings is his or her chip
count.  Thus, one could model expected winnings as a function of
the fraction of the total number of chips the player holds.  There are
two especially simple models of this.  One can use the biased random
walk of single steps to model expected winnings rather than the
probability of finishing first.  Alternatively, Chen and Ankenman,
\hyperlink{ca}{\cite{ca}} and \hyperlink{ca2}{\cite{ca2}}, propose
a model to estimate the probability of finishing first by assuming the
probability of doubling one's chips before going a broke is constant.
Again, one can instead consider expected winnings under the same
assumption.  It would be appropriate to call these the small-pot model
and the big-pot model.  They were developed with some preliminary
comparisons with data from online poker tournaments in
\hyperlink{gtg}{\cite{gtg}}.  Although skill is naturally 
incorporated into these models, neither would be appropriate
late in a poker tournament when the number of chips held by a
player's opponents becomes a critical factor in both determining the
player's expected winnings and deciding on the optimal play in each hand
of the tournament.

\section{The Independent Chip Model (ICM)}

The best-known model for tournaments between players of (roughly) equal 
abilities is the Independent Chip Model (ICM). Although it did not 
arise from any model for the movement of chips from player to player, this model is often 
used by serious poker players for analysis of the late stages of 
tournaments. 
Under the ICM, a player's probability of finishing first 
is the fraction of the total chips in play he or she holds. 
Recursively, the conditional probability of finishing in $k$th position 
given the $k-1$ players finishing 1st through $(k-1)$st is the fraction 
of the chips held by the player once the chips of the $(k-1)$ players 
finishing higher have been taken out of play.
Thus, if the fractions of chips held by players 1 through $k$ are 
$x_{1}, \ldots, x_{k}$, the probability they finish 1st through $k$th 
is
\begin{equation*}
    p_{k}(x_{1},\ldots,x_{k})=\dfrac{x_{1}x_{2}\cdots 
x_{k}}{\left(1-x_{1}\right)\left(1-x_{1}-x_{2}\right)
\cdots\left(1-x_{1}-x_{2}-\ldots-x_{k-1}\right)}.
\end{equation*}

Ganzfried and Sandholm \hyperlink{ganz}{\cite{ganz}} have used this
model to analyze the effects of position in a simulated three-player
Texas hold'em tournament under ``jam/fold'' strategies.  This author
has data from online single table tournaments that, at
first glance, suggest the ICM is a reasonable approximation of the
probabilities in the aggregate.  We mention in passing that the ICM is
essentially the model used since 1990 to determine the first three picks in the
National Basketball Association draft.

\section{Risk Aversion Under the ICM}

We introduce notation we will use in all that follows. Let 
$q_{k}(x;y;z_{1},..,z_{k})$ denote the probability, under the ICM,  
that a given player with fraction $y$ of the chips finishes first, one
with fraction $x$ finishes somewhere among the top $k+2$ players,  
with the remaining $k$ places taken by 
given players with fractions $z_{1}, \ldots, z_{k}$, who finish in this 
relative order.  Thus, by definition,
\begin{multline}\label{float}\hypertarget{float}
q_{k}(x;y;z_{1},..,z_{k})=p_{k+2}(y,x,z_{1},\ldots,z_{k})
+p_{k+2}(y,z_{1},x,\ldots,z_{k})\\+\cdots+p_{k+2}(y,z_{1},\ldots,z_{k-1},x,z_{k})
+p_{k+2}(y,z_{1},\ldots,z_{k},x).
\end{multline}

We begin with two lemmas we will use in proving both of our theorems.

\begin{lem}\label{partfrac}
    For every integer $k\ge0$,
\begin{multline*}
q_{k}(x;y;z_{1},..,z_{k})
    =\dfrac{y}{x+y}\thinspace p_{k+1}(x+y,z_{1},\ldots,z_{k})
    \\
    -p_{k+2}(y,z_{1},\ldots,z_{k},1-x-y-z_{1}-\cdots-z_{k}).
\end{multline*}    
\end{lem}  
\begin{proof}[Proof of Lemma]
We note that the second term is the probability that players with 
chip fractions $y, 
z_{1},\ldots,z_{k}$ finish in the first $k+1$ places and that anyone 
other than the player with fraction $x$ finishes in $(k+2)$nd place. 

We prove the lemma by induction. For $k=0$, we have
$$q_{0}(x;y)=\dfrac{yx}{1-y}=\dfrac{y}{x+y}(x+y)-\dfrac{y(1-x-y)}{1-y}.$$
We now assume the identity for $k$ and prove it for $k+1$.
Applying the inductive hypothesis yields
\begin{align*}
    &q_{k+1}(x;y;z_{1},\ldots,z_{k+1})\\
    &=    
    q_{k}(x;y;z_{1},\ldots,z_{k})\thinspace\dfrac{z_{k+1}}{1-x-y-z_{1}-\cdots-z_{k}}+p_{k+3}(y,z_{1},\ldots,z_{k+1},x)\\
  &=\left[\dfrac{y}{x+y}\thinspace 
  p_{k+1}(x+y,z_{1},\ldots,z_{k})\right.\\
  &-p_{k+2}(y,z_{1},\ldots,z_{k},1-x-y-z_{1}-\cdots-z_{k})\bigg]\\
  &\times 
  \dfrac{z_{k+1}}{1-x-y-z_{1}-\cdots-z_{k}}+p_{k+3}(y,z_{1},\ldots,z_{k+1},x)\\
&=\dfrac{y}{x+y}\thinspace p_{k+2}(x+y,z_{1},\ldots,z_{k+1})
-p_{k+2}(y,z_{1},\ldots,z_{k},z_{k+1})\\
&+p_{k+3}(y,z_{1},\ldots,z_{k+1},x)\\
&=\dfrac{y}{x+y}\thinspace p_{k+2}(x+y,z_{1},\ldots,z_{k+1})\\
&-p_{k+3}(y,z_{1},\ldots,z_{k+1},1-x-y-z_{1}-\cdots-z_{k+1}),
\end{align*} 
as desired.
\end{proof}

\begin{lem}\label{prodrule}
Let $f>0$ $f'\ge0$, $f''\ge0$ and $g>0$, $g'> 0$, $g''>0$. Then $fg>0$, 
$(fg)'>0$, $(fg)''>0$.    
\end{lem}  
\begin{proof}[Proof of Lemma]
The lemma is an immediate consequence of the product rule.
\end{proof}

We define a {\it fair bet} for a player to be a random variable $W$
(for wager) that is not identically 0 and whose expected value in 
chips is 0. We may let $W$ stand for either a player's gain or loss. In our context of 
poker tournaments, $W$ will be 
expressed as a fraction of the chips in play and can take on only 
finitely many values. Here we include
subsequent bets in the hand as part of the expected value.  Thus we are
interpreting the wager, which may be either initiated or accepted by the player,
to be the possible gain or loss in chips over the course of the rest of the hand.

Our first result is the following.
\begin{thm}\label{main}\hypertarget{main}
Suppose a tournament has prize money for $n$th place which is at least that
for $(n+1)$st place and that at least one player still in the
tournament will not earn as much as second place prize money.  Under
the Independent Chip Model, any fair bet in which only one other
player can gain or lose chips in the hand being played will lower the player's
expected prize money.
\end{thm}    

\begin{proof}
We will first break down the expected prize winnings under the ICM
into a sum of simpler terms, each of which is either linear or 
concave down, allowing us to conclude Theorem~\hyperlink{main}{\ref{main}} by 
convexity.  

Consider a tournament paying prize money $m_{1}\ge m_{2}\ge \ldots \ge
m_{n}$ for finishing first, second, \ldots, $n$th.  Our first
reduction is view this as $n$ simultaneous sub-tournaments, the first a
winner-take-all paying $m_{1}-m_{2}$ for first place, the second
paying $m_{2}-m_{3}$ to the first and second place finishers, through
one paying $m_{n}$ to each of the top $n$ finishers.  It will suffice
to prove that, by participating in a fair bet, a player's expected 
winnings will not increase in any of these sub-tournaments and will lose expected
value in at least one of them.

Denote the player in question by $A$ and the opponent involved in the 
bet as $B$. Let $A$ have fraction $x$ of the total number of 
chips in play, let $B$ have fraction $y$, and let $w$ denote the 
fraction of all chips $A$ loses on the bet (negative when $A$ wins). We will use 
$u_{i}$ and $z_{i}$ as needed to denote the fraction of chips held by 
other players.
 
In any sub-tournament where all players get the same prize money 
(including those with prize 0), $A$'s expected winnings are that 
amount regardless of wagers.

For the winner-take-all sub-tournament, $A$'s expected value 
participating in the wager is
$$\left(m_{1}-m_{2}\right)E[x-w]=\left(m_{1}-m_{2}\right)x,$$
i.e. $A$'s expected value hasn't changed.

All remaining sub-tournaments, of which there is at least one, satisfy 
the conditions of the theorem. Thus, it suffices to prove the 
theorem for those tournaments where each of at least two winners gets a prize of 1 
and at least one nonwinner gets 0. After losing a wager $w$, the 
probability $A$ finishes in $m$th place behind players other than $B$ 
having chip fractions $u_{1},\ldots,u_{m-1}$ is
$$p_{m}(u_{1},\ldots,u_{m-1},x-w).$$
This is linear in $w$, so 
$$E[p_{m}(u_{1},\ldots,u_{m-1},x-w)]=p_{m}(u_{1},\ldots,u_{m-1},x).$$

On the other hand, in those scenarios in which player $B$ finishes ahead 
of player $A$, with both among the winners, the dependence on $w$ is nonlinear. We partition all 
such cases by
fixing the first $m$ finishers with $B$ in $m$th place and fixing the relative positions of all other 
finishers except player $A$. Denote the fraction of chips held by 
the first $m-1$ players by $u_{1},\ldots,u_{m-1}$ and those of the 
remaining $k$ players other than $A$ or $B$ by $z_{1},\ldots,z_{k}$, 
where $k=n-m-1$. Setting $\Delta=1-u_{1}-\cdots-u_{m-1}$, player $A$'s expected winnings may be written as
$$p_{m-1}(u_{1},\ldots,u_{m-1})\cdot 
q_{k}((x-w)/\Delta;(y+w)/\Delta;z_{1}/\Delta,\ldots,z_{k}/\Delta).$$
We may drop the leading term,
$p_{m-1}(u_{1},\ldots,u_{m-1}),$
and rescale units, dividing $w$, $x$, $y$, and $z_{i}$ by
$\Delta$.  This leaves us needing to show the
concavity of the simpler expression
$q_{k}(x-w;y+w;z_{1},\ldots,z_{k})$.

Applying Lemma~\ref{partfrac}, we see that 
\begin{multline*}
    q_{k}(x-w;y+w;z_{1},\ldots,z_{k})
    =\dfrac{y+w}{x+y}\thinspace p_{k+1}(x+y,z_{1},\ldots,z_{k})\\
    -p_{k+2}(y+w,z_{1},\ldots,z_{k},1-x-y-z_{1}-\cdots-z_{k}).
\end{multline*}    
    The first term is linear in $w$. The latter expands to
$$-\dfrac{(y+w)z_{1}z_{2}\cdots 
    z_{k}(1-x-y-z_{1}-\cdots-z_{k})}{(1-y-w)(1-y-w-z_{1})\cdots(1-y-w-z_{1}-\cdots-z_{k})}.
$$
Thus it suffices to show that
\begin{equation}\label{gk}\hypertarget{gk}
    g_{k}(w)=\dfrac{y+w}{(1-y-w)(1-y-w-z_{1})\cdots(1-y-w-z_{1}-\cdots-z_{k})}
\end{equation}
has positive second derivative. Observe that, for the range of 
relevant wagers, $-(1-\min\{x,y\})\le w\le (1-\min\{x,y\})$, 
$1/(1-y-w-z_{1}-\cdots-z_{j})$ satisfies the conditions of 
Lemma~\ref{prodrule}, as does
$$g_{0}(w)=\dfrac{y+w}{1-y-w}=\dfrac1{1-y-w}-1.$$ By induction, 
$g_{k}>0$, $g'_{k}>0$, $g''_{k}>0$ for all $k$.

Therefore, $q_{k}(x-w;y+w;z_{1},\ldots,z_{k})$ is concave down and by 
convexity, 
$$E[q_{k}(x-w;y+w;z_{1},\ldots,z_{k})]<q_{k}(x;y;z_{1},\ldots,z_{k}),$$
completing the proof of Theorem~\ref{main}.
\end{proof}

Theorem~\ref{main} is false for fair wagers among three or more 
players. With many players, counterexamples are unusual. On the other 
hand, they are easy to construct: start with a tournament paying 
two places and with three players, 
all participating in a fair wager. Barring the unlikely possibility 
that expected winnings for all three are unaffected by the wager, the 
expected winnings for at least one must increase. One could easily add 
one or more uninvolved players with very small chip stacks to the 
counterexample. We give another, explicit, counterexample following 
Theorem~\ref{bystander}.

We move on to examine the impact of a fair wager on the 
expected winnings of players not involved 
in the bet.

\begin{thm}\label{bystander}\hypertarget{bystander}
Suppose a tournament has prize money for $n$th place which is at least that
for $(n+1)$st place and that at least one player still in the
tournament will not earn as much as second place prize money.  Under
the Independent Chip Model, the expected prize money of any player 
not involved in a fair bet between two players will increase.
\end{thm} 

\begin{proof}
The proof parallels that of Theorem~\ref{main}.  Let $A$ and $B$ 
denote the two players involved in the fair bet and let $C$ denote 
one player who is not involved. 

We break down $C$'s expected winnings into a sum of expected winnings
from sub-tournaments paying 1 to each of it winners.  The expected
winnings of $C$ in a winner-take-all sub-tournament is unaffected by a
fair bet.  
Similarly, they are unaffected in scenarios when neither
$A$ nor $B$ finishes ahead of $C$. It suffices to prove that $C$'s 
expected winnings increase when $B$ finishes ahead of both $A$ and 
$C$. 

For sub-tournaments paying the top two finishers, $C$'s expected 
winnings when $B$ finishes first are
$$\dfrac{(y+w)u}{1-y-w}=ug_{0}(w),$$
which we have seen is concave up. In this case we can actually 
conclude that $C$'s expected winnings must increase for any fair 
wager among two {\em or more} players.

As in the proof of Theorem~\ref{main}, we further break down to 
scenarios in which a fixed $m-1$ other players finish ahead of $B$, 
who finishes in $m$th place. Again, there is no loss of generality 
in assuming $B$ finishes first, with the top $k+2$ places paid, with 
$k\ge1$. 

Our final reduction is to scenarios 
where the order of the 
first $k$ finishers other than $A$, $B$, or $C$ are fixed.
Let $x$, $y$, and $u$ denote the 
respective fractions of all chips in play held by $A$, $B$, and $C$.
Let the fractions of the other relevant $k$ finishers be 
$z_{1},\ldots,z_{k}$. Let $w$ be the amount $B$ wins on the fair wager.
When $A$ also finishes in the money, we'll fix $C$'s position and 
``float'' $A$ as in the proof of Theorem~\ref{main}, summing over the 
different positions for $C$. To these we'll add 
those cases that $A$ does not finish in the money by floating $C$. 
Noting that only $k+2$ places are paid in this scenario, $C$'s
expected winnings are
\begin{multline*}
    \left[q_{k+1}(x-w;y+w;u,z_{1},\ldots,z_{k})-p_{k+3}(y+w,u,z_{1},\ldots,z_{k},x-w)\right]\\
    +\left[q_{k+1}(x-w;y+w;z_{1},u,z_{2},\ldots,z_{k})-p_{k+3}(y+w,z_{1},u,\ldots,z_{k},x-w)\right]\\
    +\cdots+\left[q_{k+1}(x-w;y+w;z_{1},\ldots,u,z_{k})-p_{k+3}(y+w,z_{1},\ldots,u,z_{k},x-w)\right]\\
    +q_{k}(u;y+w;z_{1},\ldots,z_{k}).
\end{multline*}    
We apply Lemma~\ref{partfrac} to each of the differences. The first is
\begin{align*}
  q_{k+1}&(x-w;y+w;u,z_{1},\ldots,z_{k})-p_{k+3}(y+w,u,z_{1},\ldots,z_{k},x-w)  \\
  =&\dfrac{y+w}{x+y}p_{k+2}(x+y,u,z_{1},\ldots,z_{k})\\
  &-p_{k+3}(y+w,u,z_{1},\ldots,z_{k},1-x-y-u-z_{1}-\cdots-z_{k})\\
  &-p_{k+3}(y+w,u,z_{1},\ldots,z_{k},x-w)\\
  =&\dfrac{y+w}{x+y}p_{k+2}(x+y,u,z_{1},\ldots,z_{k})
  -p_{k+2}(y+w,u,z_{1},\ldots,z_{k}),
\end{align*}    
with similar expressions for the other terms. From the definition of $q$, we can 
express $C$'s expected winnings as
\begin{align*}
\dfrac{y+w}{x+y}&\left[ q_{k}(u;x+y;z_{1},\ldots,z_{k})
-p_{k+2}(x+y,  z_{1},\ldots,z_{k}, u)\right]\\
&+p_{k+2}(y+w,z_{1},\ldots,z_{k},u).
\end{align*}    
The first term is linear in $w$. The second term is, 
essentially,
$$z_{1}\cdots z_{k}ug_{k+1}(w)$$ in the notation of equation (\ref{gk}) 
and is thus concave up,
completing the proof.
\end{proof}

As we saw in its proof, Theorem~\ref{bystander} holds for tournaments paying only 
two places for fair bets among three or more players, but is 
false in general.  Even then, counterexamples are quite rare.
In our counterexample below, the first three finishers win 1 unit. (Perhaps 
it's a satellite tournament to earn entry into another tournament.) 
The bet has two equally likely outcomes: A wins 16 units, B and C 
each lose 8 units or A loses 16 units, B and C each win 8 units. The 
expected winnings, to four decimal places, are given in the table below.
\vskip .3in
\begin{tabular}{|c|c|c|c|}
    \hline
    Player & Initial   & Initial  & Final 
     \\
    &   Chip Count &  Expected Winnings & 
    Expected Winnings  \\
    \hline
    A & 140 &  0.9952 &  0.9914  \\
    \hline
    B & \ \thinspace 10  &  0.5256  &  0.5316  \\
    \hline
    C & \ \thinspace 10  &  0.5256  &  0.5316    \\
    \hline
    D & \ \thinspace 50  &  0.9536  &  	0.9455  \\
    \hline
\end{tabular}
\vskip 10pt
However, there is an interesting special case, with which  we conclude this paper.
\begin{cor}\label{merge}\hypertarget{merge}
Under the assumptions of Theorem~\ref{bystander}, if two or more 
players each bet a fixed amount and the total of all bets is won by 
one of these players with probability proportional to the size of his 
or her wager, then the expected winnings of all players not involved 
in the bet increases. In 
particular, if the chips of two 
or more players are combined into a single player's stack, the 
expected winnings of all other players increase.
\end{cor} 

\begin{proof}
The case of two players is a special case of Theorem~\ref{bystander}. 
The general case can be realized as a sequence of such fair wagers between 
two players. 
\end{proof}


\begin{thebibliography}{99}


\bibitem{blt}\hypertarget{blt}
F.\ Thomas Bruss, Guy Louchard, and John W. Turner,
On the $N$-Tower-Problem and Related Problems,
Adv.\ Appl.\ Prob.\ {\bf 35:1} (2003), 278--294
(MR1975514).  

\bibitem{ca}\hypertarget{ca}
William Chen and Jerrod Ankenman,
The Theory of Doubling Up,
The Intelligent Gambler {\bf 23} (Spring/Summer 2005), 3--4. 

\bibitem{ca2}\hypertarget{ca2}
Bill Chen and Jerrod Ankenman,
The Mathematics of Poker, 
ConJelCo LLC, 2006.

\bibitem{feller}\hypertarget{feller}
William Feller, An Introduction to Probability 
Theory and Its Applications, vol.\ 1, 3rd Edition, Wiley, 1968.

\bibitem{tferg}\hypertarget{tferg}
Tom Ferguson, 
Gambler's Ruin in Three Dimensions,
1995, available at
\href{http://www.math.ucla.edu/~tom/papers/unpublished/gamblersruin.pdf}
{http://www.math.ucla.edu/\textasciitilde{}tom/papers/unpublished/gamblersruin.pdf}.

\bibitem{ganz}\hypertarget{ganz}
Sam Ganzfried and Tuomas Sandholm,
An approximate jam/fold equilibrium for 3-player no-limit Texas 
hold'em tournaments,
Proc. of 7th Int. Conf. on Autonomous Agents and
Multiagent Systems (AAMAS 2008), Padgham, Parkes, M\"uller and
Parsons (eds.), May, 12--16, 2008, Estoril, Portugal.


\bibitem{gtg}\hypertarget{gtg}
George T.\ Gilbert,
Racing Early in Tournaments,
Two Plus Two Internet Magazine,
{\bf 2:3} (March 2006). (Available from the author: g.gilbert@tcu.edu.)


\bibitem{hen}\hypertarget{hen}
Tony Henke,
Is poker different from flipping coins?,
Masters Thesis, Washington University in St. Louis, 2007, 
available at 
\href{http://economics.wustl.edu/conference/Honors/Henke\_Thesis.pdf}
{http://economics.wustl.edu/conference/Honors/Henke\_Thesis.pdf}.


\bibitem{mal}\hypertarget{mal}
Mason Malmuth, 
Settling Up in Tournaments: Part III,
in Gambling Theory and Other Topics, Two Plus Two Publishing, 1994.

\bibitem{huy}\hypertarget{huy}
Eddie Shoesmith, 
Huygens' Solution to the Gambler's Ruin Problem,
Historia Mathematica {\bf 13} (1986), 157--167 (MR0851874).

\bibitem{sw-br}\hypertarget{sw-br}
Yvik C.~Swan and F.~Thomas Bruss,
A Matrix-Analytic Approach to the $N$-Player Ruin Problem,
J.~Appl. Prob. {\bf 43} (2006), 755--766 (MR2274798).


\end{thebibliography}
\end{document}